\documentclass[a4paper,11pt,reqno]{amsart}
\usepackage{amsthm}
\usepackage{bm}
\usepackage{amsmath,amssymb,cite,mathrsfs,overpic}

\numberwithin{equation}{section}

\newtheorem{theorem}{Theorem}[section]
\newtheorem{lemma}[theorem]{Lemma}
\newtheorem{prop}[theorem]{Proposition}
\newtheorem{corollary}[theorem]{Corollary}

\theoremstyle{definition}

\newtheorem{example}[theorem]{Example}

\theoremstyle{remark}
\newtheorem{remark}[theorem]{Remark}

\numberwithin{equation}{section}

\DeclareMathOperator{\Aut}{\mathrm{Aut}(\Delta)}

\allowdisplaybreaks

\def\la{\lambda}

\def\norm#1{\lVert#1\rVert}
\def\fs{\Psi}
\def\GL{\mathrm{GL}}

\begin{document}

\title[Witten multiple zeta-functions]{On Witten multiple zeta-functions associated with semisimple Lie algebras V}
\author{Yasushi Komori}
\address{Department of Mathematics, Rikkyo University, Nishi-Ikebukuro, Toshima-ku, Tokyo 171-8501, Japan}
\email{komori@rikkyo.ac.jp}

\author{Kohji Matsumoto}
\address{Graduate School of Mathematics, Nagoya University, Chikusa-ku, Nagoya 464-8602 Japan}
\email{kohjimat@math.nagoya-u.ac.jp}

\author{Hirofumi Tsumura}
\address{Department of Mathematics and Information Sciences, Tokyo Metropolitan University, 1-1, Minami-Ohsawa, Hachioji, Tokyo 192-0397 Japan}
\email{tsumura@tmu.ac.jp}

\date{}

\maketitle

\begin{abstract}
We study the values of the zeta-function of the root system of type $G_2$ at positive
integer points.   In our previous work we considered the case when all integers are
even, but in the present paper we prove several theorems which include the situation
when some of the integers are odd.   The underlying reason why we may treat such cases
including odd integers is also discussed.
\end{abstract}

\baselineskip 19pt

\section{Introduction} \label{sec-1}

Let $\mathbb{N}$ be the set of positive integers, $\mathbb{N}_0=\mathbb{N}\cup\{0\}$,
$\mathbb{Z}$ the ring of rational integers, $\mathbb{Q}$ the rational number
field, $\mathbb{R}$ the real number field, and $\mathbb{C}$ the complex number
field.

The present paper is the continuation of our series of papers \cite{MTF,KM2,KM3,KM4}
(and also \cite{KMT,KMTpja,KMT-L,KMTJC}), in which
we have developed the theory of zeta-functions of root systems.
Motivated by the work of Witten \cite{Wi} in quantum gauge theory, Zagier \cite{Za}
defined the Witten zeta-function
$$
\zeta_W(s,{\frak g})=\sum_{\varphi}(\dim\varphi)^{-s}
$$
associated with any complex semisimple Lie algebra ${\frak g}$, where the sum
runs over all finite-dimensional irreducible representations $\varphi$ of ${\frak g}$.
The notion of zeta-functions of root systems was introduced as a multi-variable
generalization of Witten zeta-functions.   We will give the rigorous definition of
the zeta-function of the root system $\Delta$ in the next section, which we will
denote by $\zeta_r(\mathbf{s},\mathbf{y};\Delta)$, where $r$ is the rank of
$\Delta$.    By Weyl's dimension formula, it is possible to obtain the explicit
form of $\zeta_r(\mathbf{s},\mathbf{y};\Delta)$.   For example, when the root
system is of type $G_2$ and $\mathbf{y}=\mathbf{0}$, then $r=2$,
$\mathbf{s}=(s_1,s_2,s_3,s_4,s_5,s_6)\in\mathbb{C}^6$, and
\begin{align}
&\zeta_{2}(\mathbf{s};G_2)=\zeta_{2}(\mathbf{s},\mathbf{0};G_2)\notag\\
&=\sum_{m,n\geq 1}\frac{1}{m^{s_1}n^{s_2}(m+n)^{s_3}(m+2n)^{s_4}(m+3n)^{s_5}
(2m+3n)^{s_6}}. \label{G2-zeta} 
\end{align}

In our former papers, besides the general theory, we studied several individual
cases of low rank.   Zeta-functions of root systems of type $A_r$ ($r=2,3$) were 
studied in \cite{MTF,KMT,KM3}, and those of type $B_r$, $C_r$ ($r=2,3$) were
studied in \cite{KM2,KM3,KMTJC}.   Then in \cite{KM4}, the zeta-function of the
root system of type $G_2$, the simplest exceptional algebra, was discussed.

The main topic in \cite{KM4} is the situation when $s_1,\ldots,s_6$ are positive
even integers.
From our general result given in \cite[Theorem 8]{KM3}, it is possible to show that 
\begin{equation}
\zeta_{2}(2a,2b,2b,2b,2a,2a;G_2)\in \mathbb{Q}\cdot\pi^{6(a+b)}\qquad 
(a,b\in \mathbb{N}). \label{G2-Witten}
\end{equation}
Moreover the rational coefficients can be explicitly determined.
In \cite{KM4}, using the idea developed in \cite{KMTJC}, we proved certain 
functional relations (\cite[Theorem 5.1]{KM4}) which include \eqref{G2-Witten}
as special cases.

However, it is possible to treat the case when some of $s_1,\ldots,s_6$ are odd
integers.   
Zhao \cite{Zhao} expressed the values $\zeta_2({\bf k};G_2)$ for
${\bf k}\in \mathbb{N}_0^6$ (under certain conditions) in terms of double 
polylogarithms.   Using his formula, Zhao calculated numerically some of those 
special values, for example
\begin{align}\label{Zhao-num}
\zeta_2(2,1,1,1,1,1;G_2) & =  0.0099527234\cdots.                                       
\end{align}
The parity result for $\zeta_2({\bf k};G_2)$ in some extended sense has been
shown by Okamoto \cite{Ok}.   We will discuss his result more closely 
in the last section of
the present paper. 

In the present paper we also study the situation when some of $s_1,\ldots,s_6$ are 
odd integers.
In Section \ref{sec-2}, after preparing the basic notations and definitions, 
we will prove a general theorem (Theorem \ref{thm:reduction}), which gives the 
underlying reason why sometimes it is possible to evaluate the values of multiple
zeta-functions at odd integer points.
In Section \ref{sec-2.5} we will apply Theorem \ref{thm:reduction} to
$\zeta_2(\mathbf{s};G_2)$.
Sections \ref{sec-2.75} to \ref{sec-4} are devoted to the proof of functional relations
among $\zeta_2(\mathbf{s};G_2)$, the Riemann zeta-function $\zeta(s)$, and a certain
Dirichlet $L$-function.   Those functional relations especially imply explicit
evaluations of special values of $\zeta_2(\mathbf{s};G_2)$, such as
\begin{equation}                                                                        
\zeta_2(2,1,1,1,1,1;G_2) = \frac{1}{18}\zeta(2)\zeta(5)-\frac{109}{1296}\zeta(7), 
\label{Zhao-1}                                                                          
\end{equation}
(see Example \ref{E-3-3}), which we announced in \cite{KM4}.
Our result \eqref{Zhao-1} agrees with Zhao's numerical value \eqref{Zhao-num}.


\ 

\section{A general formula} \label{sec-2}

We use the same notation as in \cite{KMT,KM2,KM3,KM4}. 
We first recall several basic definitions and facts about root systems and Weyl groups
(for the details, see \cite{Hum,Hum72,Bourbaki}). 

Let $V$ be an $r$-dimensional real vector space equipped with an inner product $\langle \cdot,\cdot\rangle$.
The norm $\norm{\cdot}$ is defined by $\norm{v}=\langle v,v\rangle^{1/2}$.
The dual space $V^*$ is identified with $V$ via the inner product of $V$.
Let $\Delta$ be a finite reduced root system in $V$ which may not be irreducible, and
$\fs=\{\alpha_1,\ldots,\alpha_r\}$ its fundamental system.
We fix 
$\Delta_+$ and $\Delta_-$ as the set of all positive roots and negative roots respectively.
Then we have a decomposition of the root system $\Delta=\Delta_+\coprod\Delta_-$ .
Let $Q=Q(\Delta)$ be the root lattice, $Q^\vee$ the coroot lattice,
$P=P(\Delta)$ the weight lattice, $P^\vee$ the coweight lattice,
$P_+$ the set of integral dominant weights 
and
$P_{++}$ the set of integral strongly dominant weights
respectively defined by
\begin{align}
&  Q=\bigoplus_{i=1}^r\mathbb{Z}\,\alpha_i,\qquad
  Q^\vee=\bigoplus_{i=1}^r\mathbb{Z}\,\alpha^\vee_i,\\
& P=\bigoplus_{i=1}^r\mathbb{Z}\,\lambda_i, \qquad 
  P^\vee=\bigoplus_{i=1}^r\mathbb{Z}\,\lambda^\vee_i,\\
&   P_+=\bigoplus_{i=1}^r\mathbb{N}_0\,\lambda_i, \qquad 
  P_{++}=\bigoplus_{i=1}^r\mathbb{N}\,\lambda_i,
\end{align}
where the fundamental weights $\{\lambda_j\}_{j=1}^r$
and
the fundamental coweights $\{\lambda_j^\vee\}_{j=1}^r$
are the dual bases of $\fs^\vee$ and $\fs$
satisfying $\langle \alpha_i^\vee,\lambda_j\rangle=\delta_{ij}$ and $\langle \lambda_i^\vee,\alpha_j\rangle=\delta_{ij}$ respectively.
Let
\begin{equation}                                                                        
\rho=\frac{1}{2}\sum_{\alpha\in\Delta_+}\alpha=\sum_{j=1}^r\lambda_j                    
\label{rho-def}                                                                         
\end{equation}
be the lowest strongly dominant weight.
Then $P_{++}=P_++\rho$.

Let $\sigma_\alpha$ be the reflection with respect to a root $\alpha\in\Delta$ 
defined as
\begin{equation}                                                                        
 \sigma_\alpha:V\to V, \qquad 
\sigma_\alpha:v\mapsto v-\langle \alpha^\vee,v\rangle\alpha.
\end{equation}
For a subset $A\subset\Delta$, let
$W(A)$ be the group generated by reflections $\sigma_\alpha$ for all $\alpha\in A$. 
In particular, $W=W(\Delta)$ is the Weyl group, and
$\{\sigma_j(=\sigma_{\alpha_j})\,|\,1\leq j \leq r\}$ generates $W$.
For $w\in W$, denote
$\Delta_w=\Delta_+\cap w^{-1}\Delta_-$.

Let $\Aut$ be the subgroup of all the automorphisms $\GL(V)$ which
stabilizes $\Delta$.
Then the Weyl group $W$ is a normal subgroup of $\Aut$ and
there exists a subgroup $\Omega\subset \Aut$ such that
$\Aut=\Omega\ltimes W$.
The subgroup $\Omega$ is isomorphic to the group $\mathrm{Aut}(\Gamma)$ of automorphisms
of the Dynkin diagram $\Gamma$
(see \cite[\S 12.2]{Hum72}).

Now we can define the zeta-function of the root system $\Delta$.
For $\mathbf{s}=(s_\alpha)_{\alpha\in\Delta_+}\in\mathbb{C}^{|\Delta_+|}$
and $\mathbf{y}\in V$, it is defined by
\begin{align}                                                                          
  \label{eq:def_LZ}                                                                     
  \zeta_r(\mathbf{s},\mathbf{y};\Delta)=\sum_{\lambda\in P_{++}}                        
  e^{2\pi \sqrt{-1}\langle \mathbf{y},\lambda\rangle}                                   
  \prod_{\alpha\in\Delta_+}                                                             
  \frac{1}{\langle\alpha^\vee,\lambda\rangle^{s_\alpha}}.                               
\end{align}                                          
Let
$$
\mathcal{S}=\{\mathbf{s}=(s_{\alpha})\in\mathbb{C}^{|\Delta_+|}\;|\;
\Re s_{\alpha}>1\;{\rm for}\;\alpha\in\Delta_+\}.
$$
Then $\zeta_r(\mathbf{s},\mathbf{y};\Delta)$ is absolutely convergent in the region
$\mathcal{S}$ and is holomorphic there (\cite[Lemma 9]{KM3}).

Next, let $I\subset \{1,\ldots,r\}$ with $I\neq \emptyset$, and define a certain
linear combination $S(\mathbf{s},\mathbf{y};I;\Delta)$ of the zeta-function
associated with $I$.
Let $\fs_I=\{\alpha_i~|~i\in I\}\subset\fs$
and let $V_I$ be the linear subspace spanned by $\fs_I$.
Then $\Delta_I=\Delta\cap V_I$
is a root system in $V_I$ whose fundamental system is $\fs_I$. 
For the root system $\Delta_I$, 
we denote the corresponding
coroot lattice
 by $Q_I^\vee=\bigoplus_{i\in I}\mathbb{Z}\,\alpha_i^\vee$.
Let
\begin{gather}
  P_I=\bigoplus_{i\in I}\mathbb{Z}\,\lambda_i,\\
  P_{I+}=\bigoplus_{i\in I}\mathbb{N}_0\,\lambda_i.
\end{gather}

The natural embedding $\iota:Q_I^\vee\to Q^\vee$ induces the projection
$\iota^*:P\to P_I$. Namely for $\lambda\in P$,
$\iota^*(\lambda)$ is defined as a unique element of $P_I$ satisfying
 $\langle \iota(q),\lambda\rangle=\langle q,\iota^*(\lambda)\rangle$ for all $q\in Q_I^\vee$.
Let $W_I$ be the subgroup of $W$ generated by all the reflections associated with
the elements in $\Psi_I$, and
\begin{equation}
\label{eq:def_W_I}
  W^I=\{w\in W~|~\Delta^\vee_{I+}\subset w\Delta^\vee_+\}.
\end{equation}

For $\mathbf{s}=(s_\alpha)_{\alpha\in\Delta_+}\in\mathbb{C}^{|\Delta_+|}$,
we define an action of $\Aut$ by
\begin{equation}
  \label{eq:A_act_on_s}
  (w\mathbf{s})_\alpha=s_{w^{-1}\alpha},
\end{equation}
where we have set $s_{-\alpha}=s_\alpha$.
Now define
\begin{align}
  \label{eq:def_S}
  S(\mathbf{s},\mathbf{y};I;\Delta)
  =\sum_{\lambda\in \iota^{*-1}(P_{I+})\setminus H_{\Delta^\vee}}
  e^{2\pi \sqrt{-1}\langle \mathbf{y},\lambda\rangle}
  \prod_{\alpha\in\Delta_+}
  \frac{1}{\langle\alpha^\vee,\lambda\rangle^{s_\alpha}},
\end{align}
where $H_{\Delta^\vee}=\{v\in V~|~\langle \alpha^\vee,v\rangle=0\text{
  for some }\alpha\in\Delta\}$ is the set of all walls of Weyl
chambers. 
By \cite[Theorem 5]{KM3},
for $\mathbf{s}\in\mathcal{S}$ and $\mathbf{y}\in V$,
we have
\begin{equation}
  \label{eq:func_eq}
  S(\mathbf{s},\mathbf{y};I;\Delta)
  =
  \sum_{v\in W^I}
  \Bigl(\prod_{\alpha\in\Delta_{v^{-1}}}(-1)^{-s_{\alpha}}\Bigr)
  \zeta_r(v^{-1}\mathbf{s},v^{-1}\mathbf{y};\Delta).
\end{equation}

The following theorem implies that under a certain condition,
the number of terms on the right-hand side of \eqref{eq:func_eq} can be reduced.
\begin{theorem}
\label{thm:reduction}
  Assume that there exist $w_1\in\Aut$, $\mathbf{s}\in\mathcal{S}$ and
$\mathbf{y}\in V$ which satisfy the conditions that
  $s_\alpha\in\mathbb{Z}$ for $\alpha\in\Delta_{w_1^{-1}}$, 
  \begin{equation}
    \label{eq:as1} 
    \Bigl(\prod_{\alpha\in\Delta_{w_1^{-1}}}(-1)^{s_{\alpha}}\Bigr)=-1
  \end{equation}
  and 
  \begin{equation}
    \label{eq:as2} 
    w_1^{-1}\mathbf{s}=\mathbf{s},\qquad w_1^{-1}\mathbf{y}=\mathbf{y}.
  \end{equation}
  Then we have
  \begin{multline}\label{eq:expand0}
    S(\mathbf{s},\mathbf{y};I;\Delta)
    =
    \frac{1}{2}\Bigl(\sum_{v\in W^I \setminus w_1 W^I}
    \Bigl(\prod_{\alpha\in\Delta_{v^{-1}}}(-1)^{-s_{\alpha}}\Bigr)
    \zeta_r(v^{-1}\mathbf{s},v^{-1}\mathbf{y};\Delta) \\
    +
    \sum_{v\in W^I\setminus w_1^{-1}W^I}
    \Bigl(\prod_{\alpha\in\Delta_{v^{-1}}}(-1)^{-s_{\alpha}}\Bigr)
    \zeta_r(v^{-1}\mathbf{s},v^{-1}\mathbf{y};\Delta)\Bigr).
  \end{multline}
  Furthermore if $w_1^{-1} W^I=w_1 W^I$, then
  \begin{equation}
    \label{eq:expand1}
    S(\mathbf{s},\mathbf{y};I;\Delta)=
    \sum_{v\in W^I \setminus w_1 W^I}    
    \Bigl(\prod_{\alpha\in\Delta_{v^{-1}}}(-1)^{-s_{\alpha}}\Bigr)
    \zeta_r(v^{-1}\mathbf{s},v^{-1}\mathbf{y};\Delta).
  \end{equation}
\end{theorem}

\begin{proof}
We first prove that
\begin{equation}
\begin{split}
  \label{eq:expand2}
  &\Bigl(\prod_{\alpha\in\Delta_{w^{-1}}}(-1)^{s_{\alpha}}\Bigr)
  S(w^{-1}\mathbf{s},w^{-1}\mathbf{y};I;\Delta)\\
    &\quad=
    \sum_{v\in wW^I\cap W^I}
    \Bigl(\prod_{\alpha\in\Delta_{v^{-1}}}(-1)^{-s_{\alpha}}\Bigr)
    \zeta_r(v^{-1}\mathbf{s},v^{-1}\mathbf{y};\Delta)
    \\
    &\qquad
    +
    \Bigl(\prod_{\alpha\in\Delta_{w^{-1}}}(-1)^{s_{\alpha}}\Bigr)
    \sum_{v\in wW^I\setminus W^I}
    \Bigl(\prod_{\alpha\in\Delta_{v^{-1}w}}(-1)^{-s_{w\alpha}}\Bigr)
    \zeta_r(v^{-1}\mathbf{s},v^{-1}\mathbf{y};\Delta)
\end{split}
\end{equation}
for any $w\in\Aut$.
Since $\iota^{*-1}(P_{I+})=\bigcup_{v\in W^I}v P_+$ by 
\cite[Lemma 2]{KM3}, we can write the left-hand side of \eqref{eq:expand2} as
\begin{align}\label{tochuu1}
\Bigl(\prod_{\alpha\in\Delta_{w^{-1}}}(-1)^{s_{\alpha}}\Bigr)                        
    \sum_{u\in W^I}                                           
    \sum_{\lambda\in uP_{++}}                                                          
    e^{2\pi\sqrt{-1}\langle w^{-1}\mathbf{y},\lambda\rangle}                           
    \prod_{\alpha\in\Delta_+}                                                          
    \frac{1}{\langle\alpha^\vee,\lambda\rangle^{s_{w\alpha}}}.
\end{align}
The inner sum of \eqref{tochuu1} is equal to
\begin{align*}
\sum_{\lambda\in uP_{++}}                                                             
e^{2\pi\sqrt{-1}\langle \mathbf{y},w\lambda\rangle}                                
    \prod_{\alpha\in\Delta_+}                                                           
    \frac{1}{\langle(w\alpha)^\vee,w\lambda\rangle^{s_{w\alpha}}}
=\sum_{\lambda\in wuP_{++}}
e^{2\pi\sqrt{-1}\langle \mathbf{y},\lambda\rangle}
   \prod_{\alpha\in w\Delta_+} 
   \frac{1}{\langle\alpha^\vee,\lambda\rangle^{s_{\alpha}}},
\end{align*}
and concerning the last product,
since $w\Delta_+\cap\Delta_-=w\Delta_w=-\Delta_{w^{-1}}$, we have
\begin{align}\label{signature}
\prod_{\alpha\in w\Delta_+}                                                            
   \frac{1}{\langle\alpha^\vee,\lambda\rangle^{s_{\alpha}}}
=\Bigl(\prod_{\alpha\in \Delta_{w^{-1}}}(-1)^{-s_{\alpha}}\Bigr)
   \prod_{\alpha\in \Delta_+}
   \frac{1}{\langle\alpha^\vee,\lambda\rangle^{s_{\alpha}}}.
\end{align}
Using this expression when $wu\in W^I$, we see that \eqref{tochuu1} is equal to
\begin{equation}\label{tochuu2}
\begin{split}
 &\sum_{\substack{u\in W^I\\wu\in W^I}}                                                
    \sum_{\lambda\in wuP_{++}}                                                         
    e^{2\pi\sqrt{-1}\langle \mathbf{y},\lambda\rangle}                                 
    \prod_{\alpha\in\Delta_+}                                                          
    \frac{1}{\langle\alpha^\vee,\lambda\rangle^{s_\alpha}}                             
    \\                                                                                 
    &\qquad+                                                                           
    \Bigl(\prod_{\alpha\in\Delta_{w^{-1}}}(-1)^{s_{\alpha}}\Bigr)                      
    \sum_{\substack{u\in W^I\\wu\notin W^I}}                                           
    \sum_{\lambda\in uP_{++}}                                                          
    e^{2\pi\sqrt{-1}\langle w^{-1}\mathbf{y},\lambda\rangle}                           
    \prod_{\alpha\in\Delta_+}                                                          
    \frac{1}{\langle\alpha^\vee,\lambda\rangle^{s_{w\alpha}}}\\
    &=\Sigma_1+\Bigl(\prod_{\alpha\in\Delta_{w^{-1}}}(-1)^{s_{\alpha}}\Bigr)
    \Sigma_2,
\end{split}
\end{equation}
say.   Putting $wu=v$, we have
\begin{equation*}  
\begin{split}
\Sigma_1&=\sum_{v\in wW^I\cap W^I}
   \sum_{\lambda\in vP_{++}}                                                            
    e^{2\pi\sqrt{-1}\langle \mathbf{y},\lambda\rangle}    
    \prod_{\alpha\in\Delta_+}                                                           
    \frac{1}{\langle\alpha^\vee,\lambda\rangle^{s_\alpha}}\\
&=\sum_{v\in wW^I\cap W^I}                                                        
   \sum_{\lambda\in P_{++}}                               
    e^{2\pi\sqrt{-1}\langle \mathbf{y},v\lambda\rangle}                                  
    \prod_{\alpha\in\Delta_+}                                                           
    \frac{1}{\langle\alpha^\vee,v\lambda\rangle^{s_\alpha}}\\ 
&=\sum_{v\in wW^I\cap W^I}                                                              
   \sum_{\lambda\in P_{++}}
   e^{2\pi\sqrt{-1}\langle v^{-1}\mathbf{y},\lambda\rangle}
   \prod_{\alpha\in v^{-1}\Delta_+}
   \frac{1}{\langle\alpha^\vee,\lambda\rangle^{s_{v\alpha}}},
\end{split}
\end{equation*}
and, as in \eqref{signature}, the signature appears from the last product when
$\alpha\in v^{-1}\Delta_+\cap \Delta_- =-\Delta_v$, which is
$$
\prod_{\alpha\in\Delta_v}(-1)^{-s_{v\alpha}}
   =\prod_{\alpha\in\Delta_{v^{-1}}}(-1)^{-s_{\alpha}}
$$
because $v\Delta_v=-\Delta_{v^{-1}}$.
Therefore we obtain
\begin{equation}\label{tochuu3}
\begin{split}
\Sigma_1&=\sum_{v\in wW^I\cap W^I}                          
   \Bigl(\prod_{\alpha\in\Delta_{v^{-1}}}(-1)^{-s_{\alpha}}\Bigr)
   \sum_{\lambda\in P_{++}}                                                             
   e^{2\pi\sqrt{-1}\langle v^{-1}\mathbf{y},\lambda\rangle}                             
   \prod_{\alpha\in \Delta_+}                                                     
   \frac{1}{\langle\alpha^\vee,\lambda\rangle^{s_{v\alpha}}}\\
&=\sum_{v\in wW^I\cap W^I}                 
   \Bigl(\prod_{\alpha\in\Delta_{v^{-1}}}(-1)^{-s_{\alpha}}\Bigr)
   \zeta_r(v^{-1}\mathbf{s},v^{-1}\mathbf{y};\Delta).
\end{split}
\end{equation}
Similarly we can show
\begin{align}\label{tochuu4}
\Sigma_2=\sum_{v\in wW^I\setminus W^I}\Bigl(\prod_{\alpha\in\Delta_{v^{-1}w}}
(-1)^{-s_{w\alpha}}\Bigr)\zeta_r(v^{-1}\mathbf{s},v^{-1}\mathbf{y};\Delta).
\end{align}
Substituting \eqref{tochuu3} and \eqref{tochuu4} into \eqref{tochuu2}, we
obtain \eqref{eq:expand2}.

Now consider the equation
\begin{equation}\label{tochuu5}
  2S(\mathbf{s},\mathbf{y};I;\Delta)=
  S(\mathbf{s},\mathbf{y};I;\Delta)
  -
  \Bigl(\prod_{\alpha\in\Delta_{w_1^{-1}}}(-1)^{s_{\alpha}}\Bigr)
  S(w_1^{-1}\mathbf{s},w_1^{-1}\mathbf{y};I;\Delta),
\end{equation}
which trivially follows from the assumptions \eqref{eq:as1} and \eqref{eq:as2}.
Substitute the expansions \eqref{eq:func_eq} and \eqref{eq:expand2}
(with $w=w_1$) to the
right-hand side of \eqref{tochuu5}.   The first sum on the right-hand side of 
\eqref{eq:expand2} is cancelled with the part $v\in w_1 W^I\cap W^I$ of
\eqref{eq:func_eq}, and hence 
\begin{equation}\label{tochuu6}
  \begin{split}
&2S(\mathbf{s},\mathbf{y};I;\Delta)=
    \sum_{v\in W^I \setminus w W^I}
    \Bigl(\prod_{\alpha\in\Delta_{v^{-1}}}(-1)^{-s_{\alpha}}\Bigr)
    \zeta_r(v^{-1}\mathbf{s},v^{-1}\mathbf{y};\Delta)
    \\
    &\qquad
    -
\Bigl(\prod_{\alpha\in\Delta_{w^{-1}}}(-1)^{s_{\alpha}}\Bigr)                         
    \sum_{v\in wW^I\setminus W^I}                           
    \Bigl(\prod_{\alpha\in\Delta_{v^{-1}w}}(-1)^{-s_{w\alpha}}\Bigr)                     
    \zeta_r(v^{-1}\mathbf{s},v^{-1}\mathbf{y};\Delta).
\end{split}                                                                             
\end{equation}
We see that the second term on the right-hand side of \eqref{tochuu6} is, 
renaming $w^{-1}v$ by $v$ and using \eqref{eq:as1} and \eqref{eq:as2}, 
equal to
\begin{equation}\label{tochuu7}                                                                        
  \begin{split}
    &\Bigl(\prod_{\alpha\in\Delta_{w^{-1}}}(-1)^{s_{\alpha}}\Bigr)
    \sum_{v\in W^I\setminus w^{-1}W^I}
    \Bigl(\prod_{\alpha\in\Delta_{v^{-1}}}(-1)^{-s_{w\alpha}}\Bigr)
    \zeta_r(v^{-1}w^{-1}\mathbf{s},v^{-1}w^{-1}\mathbf{y};\Delta)\\
&\quad=-
    \sum_{v\in W^I\setminus w^{-1}W^I}
    \Bigl(\prod_{\alpha\in\Delta_{v^{-1}}}(-1)^{-s_{\alpha}}\Bigr)
    \zeta_r(v^{-1}\mathbf{s},v^{-1}\mathbf{y};\Delta).
  \end{split}
\end{equation}
The desired results follow from \eqref{tochuu6} and \eqref{tochuu7}.
\end{proof}

\begin{remark}\label{rem1}
The above Theorem \ref{thm:reduction} is stated under the condition 
$\mathbf{s}\in\mathcal{S}$.
Treating more carefully, however, we can generalize this theorem
to the case when $s_{\alpha}=1$ for some of the $\alpha$'s (cf. 
\cite[Remark 2]{KM3}).
\end{remark}

\begin{remark}\label{rem2}
Since the right-hand sides of \eqref{eq:expand0} and \eqref{eq:expand1}
include signature factors, sometimes the right-hand side might be zero.
If so, then Theorem \ref{thm:reduction} gives no useful information.
In the next section we will give examples when the right-hand side does not
vanish. 
This is the key point why we can sometimes treat the situation when some of the
variables are odd integers.
\end{remark}

\

\section{Application of Theorem \ref{thm:reduction} to the case $G_2$} \label{sec-2.5}

Hereafter in the present paper we concentrate on the study of the zeta-function of
the root system $G_2$.   The fundamental system of $G_2$ is 
$\Psi=\{\alpha_1,\alpha_2\}$, where $|\alpha_2|=\sqrt{3}|\alpha_1|$ and the angle 
between $\alpha_1$ and $\alpha_2$ is $5\pi/6$. 
Denote the positive roots by $\alpha_1,\ldots,\alpha_6$, where
\begin{equation}                                                                        
  \begin{aligned}                                                                       
    \alpha_3&=3\alpha_1+\alpha_2,\qquad &\alpha_3^\vee&=\alpha_1^\vee+\alpha_2^\vee\\   
    \alpha_4&=3\alpha_1+2\alpha_2,\qquad &\alpha_4^\vee&=\alpha_1^\vee+2\alpha_2^\vee\\ 
    \alpha_5&=\alpha_1+\alpha_2,\qquad &\alpha_5^\vee&=\alpha_1^\vee+3\alpha_2^\vee\\   
    \alpha_6&=2\alpha_1+\alpha_2,\qquad &\alpha_6^\vee&=2\alpha_1^\vee+3\alpha_2^\vee   
  \end{aligned}                                                                         
\end{equation}
and we abbreviate $\sigma_j=\sigma_{\alpha_j}$.
Applying Weyl's dimension formula with the above data to \eqref{eq:def_LZ}, we 
find that the form of the zeta-function of $G_2$ is given by \eqref{G2-zeta},
with $s_j=s_{\alpha_j}$ ($1\leq j\leq 6$).

Now we show an application of Theorem \ref{thm:reduction} to the $G_2$ case.
Assume $s_i$ are all positive integers ($\geq 2$).
Let $I=\{2\}$, $\mathbf{y}=\mathbf{0}$ and $w_1=w_0\sigma_1$ with
the longest element
\begin{equation*}
  w_0=\sigma_1\sigma_2\sigma_1\sigma_2\sigma_1\sigma_2=
\sigma_2\sigma_1\sigma_2\sigma_1\sigma_2\sigma_1=-1.
\end{equation*}
Then we have $w_1=w_1^{-1}=
\sigma_2\sigma_1\sigma_2\sigma_1\sigma_2=-\sigma_1$, so
$w_1\alpha_1=\alpha_1$, $w_1\alpha_2=-\alpha_3$, $w_1\alpha_3=-\alpha_2$,
$w_1\alpha_4=-\alpha_4$, $w_1\alpha_5=-\alpha_6$, $w_1\alpha_6=-\alpha_5$,
and hence $w_1^{-1}\mathbf{s}=(s_1,s_3,s_2,s_4,s_6,s_5)$.
Therefore, when $s_2=s_3$ and $s_5=s_6$, we have $w_1^{-1}\mathbf{s}=\mathbf{s}$.
Since
\begin{equation*}                                                                       
  \Delta_{w_1^{-1}}=\Delta_+\cap w_1\Delta_-=\Delta_+\cap \sigma_1\Delta_+=\Delta_
+\setminus\{\alpha_1\},                                                                 
\end{equation*}
we have
\begin{equation*}                                                                       
    \Bigl(\prod_{\alpha\in\Delta_{w_1^{-1}}}(-1)^{s_{\alpha}}\Bigr)=                      
(-1)^{s_2+s_3+s_4+s_5+s_6}.                                                             
\end{equation*}
Therefore we can apply Theorem \ref{thm:reduction} when $s_2=s_3$, $s_5=s_6$, and
$s_2+s_3+s_4+s_5+s_6$ is odd.
It is easy to see that
\begin{gather*}
  W_I=\{1,\sigma_2\}, \\
  W^I=\{1,\sigma_1,\sigma_1\sigma_2,
  \sigma_1\sigma_2\sigma_1, \sigma_1\sigma_2\sigma_1\sigma_2,
  \sigma_1\sigma_2\sigma_1\sigma_2\sigma_1\},\\
  w_1 W^I=\{
  \sigma_2\sigma_1\sigma_2\sigma_1\sigma_2,
  \sigma_1\sigma_2\sigma_1\sigma_2\sigma_1\sigma_2,
  \sigma_1\sigma_2\sigma_1\sigma_2\sigma_1,
  \sigma_1\sigma_2\sigma_1\sigma_2,
  \sigma_1\sigma_2\sigma_1,
  \sigma_1\sigma_2\},
\end{gather*}
and hence $ W^I\setminus w_1 W^I=\{1,\sigma_1\}$.
Thus by \eqref{eq:expand1}, we have
$$
S(\mathbf{s},\mathbf{0};\{2\};G_2)=\zeta_2(\mathbf{s},\mathbf{0};G_2)
+(-1)^{s_1}\zeta_2(\sigma_1^{-1}\mathbf{s},\mathbf{0};G_2).
$$
Since $s_2=s_3$ and $s_5=s_6$, we have $\sigma_1^{-1}\mathbf{s}=\mathbf{s}$.
Therefore if $s_1$ is even then the right-hand side of the above is
$2\zeta_2(\mathbf{s},\mathbf{0};G_2)$.   The conclusion is as follows.

\begin{prop}\label{Prop-2.2}
Let $p,q,r,u\in\mathbb{N}_{\geq 2}$ with even $p$ and odd $r$. Then
\begin{equation}
  S((p,q,q,r,u,u),\mathbf{0};\{2\};G_2)=2\zeta_2(p,q,q,r,u,u;G_2). \label{2-29}
\end{equation}
\end{prop}

Similarly we can treat the case $I=\{1\}$, $\mathbf{y}=\mathbf{0}$.
Then $W_I=\{1,\sigma_1\}$ and
\begin{equation}\label{doumodoumo}
W^I=\{1,\sigma_2,\sigma_2\sigma_1,\sigma_2\sigma_1\sigma_2,
\sigma_2\sigma_1\sigma_2\sigma_1, \sigma_2\sigma_1\sigma_2\sigma_1\sigma_2\}.
\end{equation}
In this case we choose $w_1=w_0\sigma_2$.   Then
$W^I\setminus w_1W^I=\{1,\sigma_2\}$.
We can apply Theorem \ref{thm:reduction} when $s_1=s_5$, $s_3=s_4$, and
$s_1+s_3+s_4+s_5+s_6$ is odd.   We obtain

\begin{prop}\label{Prop-2.3}
Let $p,q,r,u\in\mathbb{N}_{\geq 2}$ with even $q$ and odd $u$. Then
\begin{equation}
  S((p,q,r,r,p,u),\mathbf{0};\{1\};G_2)=2\zeta_2(p,q,r,r,p,u;G_2). \label{2-30}
\end{equation}
\end{prop}

\

\section{A functional relation corresponding to $I=\{1\}$} \label{sec-2.75}

The results in the previous sections are valid only in the case
$\mathbf{s}\in\mathcal{S}$.   Hereafter we study the situation which includes the
case when some of the variables take the value 1.

In this section we will show a functional relation which corresponds, in some sense,
to the case $I=\{1\}$ in the preceding section.   The discussion on the general
situation would require more pages, so here we restrict ourselves to the following
one special example. 

\begin{example}\label{Exam-i-1}
The functional relation
\begin{align}
& \zeta_2(s,2,1,1,1,1;G_2)+\zeta_2(s,1,2,1,1,1;G_2)+\zeta_2(1,2,1,1,s,1;G_2)\notag\\
& \qquad -\zeta_2(1,1,2,1,1,s;G_2)+\zeta_2(1,1,1,2,1,s;G_2)-\zeta_2(1,1,1,2,s,1;G_2)\notag\\
& =\zeta(2)\zeta(s+4) - \left(\frac{651}{8}-2^{-s-1}-\frac{5\cdot 3^{-s-2}}{2}\right)\zeta(s+6) \notag\\
& \qquad  + \frac{9\pi }{2}\sum_{m\geq 1}\frac{\sin(2\pi m/3)}{m^{s+5}} -{135}\sum_{m\geq 1}\frac{\cos(2\pi m/3)}{m^{s+6}} \notag\\
& =\zeta(2)\zeta(s+4) - \left(\frac{111}{8}-2^{-s-1}\right)\zeta(s+6) + \frac{81}{4}L(1,\chi_3)L(s+5,\chi_3), \label{FR-01}
\end{align}
holds for $s\in \mathbb{C}$ except for singularities on the both sides, 
where $L(\cdot,\chi_3)$ denotes the Dirichlet $L$-function attached to the 
primitive Dirichlet character $\chi_3$ of conductor $3$.
\end{example}

\begin{proof}
We calculate
\begin{equation}\label{S-no-siki}
S((s,2,1,1,1,1),\mathbf{0},\{1\};G_2)
\end{equation}
in two ways.   Using \eqref{eq:func_eq} and \eqref{doumodoumo}, we find
that \eqref{S-no-siki} is equal to the left-hand side of \eqref{FR-01}.
On the other hand, since $P_{\{1\}+}=\mathbb{N}_0 \lambda_1$, from \eqref{eq:def_S}
it follows that \eqref{S-no-siki} is equal to
\begin{align}\label{uhen}
\sum_{m\geq 1}\sum_{n\not=0 \atop {m+n\not=0 \atop {m+2n\not=0 \atop {m+3n\not=0 \atop 2m+3n\not=0}}}}\frac{1}{m^s n^2 (m+n) (m+2n) (m+3n) (2m+3n)}.
\end{align}
Therefore the remaining task is to show that \eqref{uhen} is equal to the
right-hand side of \eqref{FR-01}.
A direct way of the proof is to rewrite \eqref{uhen} as
\begin{align}
& \sum_{m\geq 1}\sum_{n\not=0 \atop {l_1\not=0 \atop {l_2\not=0 \atop {l_3\not=0 \atop l_4\not=0}}}}\frac{1}{m^s n^2 l_1l_2l_3l_4}\int_{0}^{1}\int_{0}^{1}\int_{0}^{1}\int_{0}^{1}e^{2\pi ix_1(m+n-l_1)}e^{2\pi ix_2(m+2n-l_2)}\notag\\
& \ \qquad \times e^{2\pi ix_3(m+3n-l_3)}e^{2\pi ix_4(2m+3n-l_4)}dx_1dx_2dx_3dx_4
\label{eq-fin}
\end{align}
and compute this by using
\begin{equation}
\lim_{M\to \infty}\sum_{m=-M}^{M}\frac{e^{2\pi i m\theta}}{m^k}=-\frac{(2\pi i)^k}{k!}B_k(\theta-[\theta])\ \ (k\in \mathbb{N};\,\theta\in \mathbb{R}) \label{Lerch}
\end{equation}
(\cite[Theorem 12.19]{Apostol}),
where $[\theta]$ is the integer part of $\theta$, and $\{ B_n(x)\}$ are the Bernoulli polynomials defined by 
\begin{equation}
\frac{te^{xt}}{e^t-1}=\sum_{n=0}^\infty B_n(x)\frac{x^n}{n!}.\label{Ber-poly}
\end{equation}
This method is
based on an idea initiated by Zagier \cite{Za2} and systematically used by
Nakamura \cite{Nak06, Nak08a, Nak08b} (and also in \cite{MNOT, MNT}).
However, if we follow this way, the necessary computations are really enormous.
Therefore, in order to reduce the total amount of computations, we first modify
\eqref{uhen} using partial fraction decompositions, before applying the idea of
Zagier-Nakamura.

First, using the partial fraction decomposition
$$
\frac{1}{(m+3n)(2m+3n)}=\frac{1}{m(m+3n)}-\frac{1}{m(2m+3n)},
$$
we divide \eqref{uhen} into two sums
\begin{align*}
{\sum\sum}^* \frac{1}{m^{s+1} n^2 (m+n) (m+2n) (m+3n)}
-{\sum\sum}^* \frac{1}{m^{s+1} n^2 (m+n) (m+2n) (2m+3n)},
\end{align*}
where ${\sum\sum}^*$ denotes the same double sum as in \eqref{uhen}.
We then apply the same type of partial fraction decompositions some more times
to find that \eqref{uhen} is equal to
\begin{align}\label{prex1}
&2{\sum\sum}^*\frac{1}{m^{s+1}n^3(m+n)(m+2n)}
   -\frac{5}{2}{\sum\sum}^*\frac{1}{m^{s+1}n^4(m+n)}\\
&\;   +\frac{1}{2}{\sum\sum}^*\frac{1}{m^{s+1}n^4(m+3n)} 
   +4{\sum\sum}^*\frac{1}{m^{s+1}n^4(2m+3n)} \notag\\
&=2\Sigma_1-\frac{5}{2}\Sigma_2+\frac{1}{2}\Sigma_3+4\Sigma_4,\notag
\end{align}
say.
We divide $\Sigma_1$ as
$$
\Sigma_1=
{\sum\sum}^{**}-{\sum\sum_{\!\!\!\!\!\!\!\!\!\!\!m+3n= 0}}^{**}
-{\sum\sum_{\!\!\!\!\!\!\!\!\!\!\!2m+3n= 0}}^{**},
$$
where ${\sum\sum}^{**}$ denotes the sum over $m\geq 1$, $n\neq 0$, $m+n\neq 0$
and $m+2n\neq 0$.   Denote the first term by $\Sigma_{11}$. 
Putting $m=3l$ and $n=-l$ in the second sum, we see that the second term is
$$
-\sum_{l=1}^{\infty}\frac{1}{(3l)^{s+1}(-l)^3(3l-l)(3l-2l)}=\frac{1}{2\cdot 3^{s+1}}
\zeta(s+6).
$$
Similarly the third term is $-3^{-s-1}2^{-3}\zeta(s+6)$.   Therefore we have
\begin{equation}\label{prex2}
\Sigma_1=\Sigma_{11}+\left(\frac{1}{2\cdot 3^{s+1}}-\frac{1}{2^3 3^{s+1}}\right)
\zeta(s+6)=
\Sigma_{11}+\frac{1}{2^3 3^s}\zeta(s+6).
\end{equation}
As for $\Sigma_2$, we divide it as
\begin{equation}\label{prex3}
\Sigma_2={\sum\sum}^{***}-{\sum\sum_{\!\!\!\!\!\!\!\!\!\!\!m+2n= 0}}^{***}
-{\sum\sum_{\!\!\!\!\!\!\!\!\!\!\!m+3n= 0}}^{***}
-{\sum\sum_{\!\!\!\!\!\!\!\!\!\!\!2m+3n= 0}}^{***},
\end{equation}
where ${\sum\sum}^{***}$ denotes the sum over $m\geq 1$, $n\neq 0$, $m+n\neq 0$.
Denote the first term by $\Sigma_{21}$ and evaluate the remaining three sums as above to obtain
\begin{equation}\label{prex4}                                                           
\Sigma_2=\Sigma_{21}-\left(\frac{1}{2^{s+1}}-\frac{1}{2^4 3^{s-1}}\right)        
\zeta(s+6).                                                                             
\end{equation}
Similarly
\begin{equation}\label{prex5}                                                           
\Sigma_3=\Sigma_{31}+\left(\frac{1}{2}+\frac{1}{2^{s+1}}+\frac{1}{2^4 3^{s+2}}\right)
\zeta(s+6)                                                                             
\end{equation}
and
\begin{equation}\label{prex6}                                                           
\Sigma_4=\Sigma_{41}+\left(1-\frac{1}{2^{s+1}}-\frac{1}{3^{s+2}}\right)              
\zeta(s+6),                                                                            
\end{equation}
where
$$
\Sigma_{31}=\sum_{m\geq 1}\sum_{n\not=0 \atop {m+3n\not=0}}\frac{1}
{m^{s+1}n^4(m+3n)},\quad
\Sigma_{41}=\sum_{m\geq 1}\sum_{n\not=0 \atop {2m+3n\not=0}}\frac{1}
{m^{s+1}n^4(2m+3n)}.
$$
Applying a partial fraction decomposition once more, we obtain
$$
\Sigma_{11}={\sum\sum}^{**}\frac{1}{m^{s+1}n^4(m+n)}-
{\sum\sum}^{**}\frac{1}{m^{s+1}n^4(m+2n)}.
$$
On the right-hand side, we separate the part $m+2n=0$ from the first double sum and
separate the part $m+n=0$ from the second double sum.   We obtain
\begin{equation}\label{prex7}
\Sigma_{11}=\Sigma_{21}-\Sigma^{\sharp}
-\left(1+\frac{1}{2^{s+1}}\right)\zeta(s+6),
\end{equation}
where
$$
\Sigma^{\sharp}=\sum_{m\geq 1}\sum_{n\neq 0 \atop{m+2n\neq 0}}
\frac{1}{m^{s+1}n^4(m+2n)}.
$$

We evaluate $\Sigma_{21}$.   Recall the definition of the zeta-function of the
root system $A_2$ (or the Mordell-Tornheim double sum)
$$
\zeta_2(s_1,s_2,s_3;A_2)=\sum_{m,n\geq 1}\frac{1}{m^{s_1}n^{s_2}(m+n)^{s_3}}.
$$
The part corresponding to positive $n$ of the sum $\Sigma_{21}$ is exactly 
$\zeta_2(s+1,4,1;A_2)$.   The part corresponding to negative $n$ is, putting
$m-n=l$ when $m>n$ and $n-m=k$ when $m<n$, equal to
\begin{align*}
&\sum_{n,l\geq 1}\frac{1}{(n+l)^{s+1}n^4 l}-\sum_{m,k\geq 1}\frac{1}
{m^{s+1}(m+k)^4 k}\\
&=\zeta_2(4,s+1,1;A_2)-\zeta_2(1,s+1,4;A_2).
\end{align*}
Therefore
\begin{align}
\Sigma_{21}&=\zeta_2(s+1,4,1;A_2)+\zeta_2(4,s+1,1;A_2)-\zeta_2(1,s+1,4;A_2)\notag\\
&=-5\zeta(s+6)+2\zeta(2)\zeta(s+4)+2\zeta(4)\zeta(s+2),\label{prex9}
\end{align}
where the second equality can be seen by \cite[Theorem 3.1]{KMTJC}.

As for $\Sigma^{\sharp}$, we apply the method of Zagier-Nakamura.   Write
$\Sigma^{\sharp}$ as
\begin{align}
\Sigma^{\sharp}&=\sum_{m\geq 1 \atop {n\neq 0 \atop {l\neq 0}}}\frac{1}
{m^{s+1}n^4 l}\int_0^1 e^{2\pi i(m+2n-l)\theta}d\theta\notag\\
&=\sum_{m\geq 1}\frac{1}{m^{s+1}}\int_0^1 e^{2\pi im\theta}\sum_{n\neq 0}
\frac{e^{2\pi in\cdot 2\theta}}{n^4}\sum_{l\neq 0}\frac{e^{2\pi il(-\theta)}}{l}d\theta,
\label{prex10}
\end{align}
and apply \eqref{Lerch}.   We obtain
\begin{equation}\label{prex11}
\Sigma^{\sharp}=\frac{(2\pi i)^5}{24}\sum_{m\geq 1}\frac{1}{m^{s+1}}(J_1+J_2),
\end{equation}
where
$$
J_1=\int_0^{1/2}e^{2\pi im\theta}B_4(2\theta)B_1(1-\theta)d\theta,\quad
J_2=\int_{1/2}^{1}e^{2\pi im\theta}B_4(2\theta-1)B_1(1-\theta)d\theta.
$$
Since $B_1(x)=x-(1/2)$ and $B_4(x)=x^4-2x^3+x^2-(1/30)$, the factors
$B_4(2\theta)B_1(1-\theta)$ and $B_4(2\theta-1)B_1(1-\theta)$ are polynomials in
$\theta$ of degree 5.   It is easy to see recursively that
\begin{equation}\label{prex12}
\int_0^{1/2}e^{2\pi im\theta} \theta^k d\theta=\sum_{j=1}^{k+1}
\frac{(-1)^{j-1+m}k!}{(2\pi im)^j 2^{k+1-j}(k+1-j)!}-\frac{(-1)^k k!}{(2\pi im)^{k+1}}
\end{equation}
and
\begin{equation}\label{prex13}
\int_{1/2}^{1}e^{2\pi im\theta} \theta^k d\theta=\sum_{j=1}^{k+1}
\frac{(-1)^{j-1}k!}{(2\pi im)^j (k+1-j)!}\left(1-\frac{(-1)^m}{2^{k+1-j}}\right).
\end{equation}
Using these formulas, we can evaluate $J_1$ and $J_2$.   Substituting the results
into \eqref{prex11}, we find that $\Sigma^{\sharp}$ can be written in terms of
$\zeta(s)$ and $\phi(s)=\sum_{m=1}^{\infty}(-1)^{m}m^{-s}=(2^{1-s}-1)\zeta(s)$,
more explicitly,
\begin{equation}\label{prex14}
\Sigma^{\sharp}=\frac{\pi^4}{45}\zeta(s+2)+\frac{4\pi^2}{3}\zeta(s+4)
   -\left(16+\frac{1}{2^s}\right)\zeta(s+6).
\end{equation}
Substituting \eqref{prex9} and \eqref{prex14} into \eqref{prex7} we obtain
\begin{equation}\label{prexx}
\Sigma_{11}=-\pi^2\zeta(s+4)+\left(10+\frac{1}{2^{s+1}}\right)\zeta(s+6),
\end{equation}
and so
\begin{equation}\label{prex15}
2\Sigma_{11}-\frac{5}{2}\Sigma_{21}=-\frac{\pi^4}{18}\zeta(s+2)
-\frac{17\pi^2}{6}\zeta(s+4)+\left(\frac{65}{2}+\frac{1}{2^s}\right)\zeta(s+6).
\end{equation}

The evaluation of $\Sigma_{31}$ and $\Sigma_{41}$ is similar to that of
$\Sigma^{\sharp}$.   In these cases, instead of \eqref{prex12} and \eqref{prex13},
the integrals over the intervals $[0,1/3]$, $[1/3,2/3]$, and $[2/3,1]$ appear,
and hence the 3rd root of unity appears.   We obtain
\begin{align}
&\frac{1}{2}\Sigma_{31}+4\Sigma_{41}
=\frac{\pi^4}{18}\zeta(s+2)+3\pi^2\zeta(s+4)+
\left(-\frac{911}{8}-\frac{1}{2^{s+1}}+\frac{5}{2\cdot 3^{s+2}}\right)\zeta(s+6)\notag\\
&\qquad+\frac{9\pi}{2}\sum_{m\geq 1}\frac{\sin(2\pi m/3)}{m^{s+5}}
-135\sum_{m\geq 1}\frac{\cos(2\pi m/3)}{m^{s+6}}.\label{prex16}
\end{align}
Moreover it is easy to see that
\begin{align}\label{cosclausen}
\sum_{m\geq 1}\frac{\cos(2\pi m/3)}{m^{s+6}}=\frac{3^{-s-5}-1}{2}\zeta(s+6),
\end{align}
\begin{align}\label{sinclausen}
\sum_{m\geq 1}\frac{\sin(2\pi m/3)}{m^{s+5}}=
\frac{\sqrt{3}}{2}\sum_{m\geq 1}\frac{\chi_3(m)}{m^{s+5}}
=\frac{\sqrt{3}}{2}L(s+5,\chi_3).  
\end{align}
Consequently we can conclude that \eqref{uhen} coincides with 
the right-hand side of \eqref{FR-01}. 
\end{proof}

In particular, setting $s=1$ in \eqref{FR-01}, we obtain that the left-hand side is equal to $2\zeta_2(1,2,1,1,1,1;G_2)$ (see \eqref{2-29}). Hence we have
\begin{align}
\zeta_2(1,2,1,1,1,1;G_2)
=\frac{1}{2}\zeta(2)\zeta(5)- \frac{109}{16}\zeta(7)+ \frac{81}{8}L(1,\chi_3)L(6,\chi_3). \label{VR-01}
\end{align}

\begin{remark}\label{digression}
A little digression.   Recall that the zeta-function of the root system $C_2$ is
defined by
$$
\zeta_2(s_1,s_2,s_3,s_4;C_2)=\sum_{m,n\geq 1}\frac{1}{m^{s_1}n^{s_2}
(m+n)^{s_3}(m+2n)^{s_4}}.
$$
Divide $\Sigma_{11}$ into two subsums according as $n\geq 1$ and $n\leq -1$.
Then the former part is exactly $\zeta_2(s+1,3,1,1;C_2)$.
The latter is further divided according as $m-n>0$ and $m-n<0$.   The part
corresponding to $m-n<0$ is $-\zeta_2(s+1,1,3,1;C_2)$, while the remaining part
is again divided into two subsums.   The conclusion is that
\begin{align}
\Sigma_{11}=&\zeta_2(s+1,3,1,1;C_2)-\zeta_2(s+1,1,3,1;C_2)\notag\\
&+\zeta_2(1,1,3,s+1;C_2)-\zeta_2(1,3,1,s+1;C_2).\label{dig-1}
\end{align}
On the other hand, we have shown that $\Sigma_{11}$ can be written in terms of
$\zeta(s)$ (see \eqref{prexx}).   Combining these two formulas \eqref{dig-1} and
\eqref{prexx}, we obtain a functional relation between the zeta-function of
$C_2$ and the Riemann zeta-function, which is different from the previously known
relations (\cite[Section 8]{KMTJC}, \cite[Section 5]{Nak08b}).
\end{remark}
\ 

\section{Some lemmas} \label{sec-3}

In the next section, we will deduce a functional relation corresponding to
the case $I=\{2\}$, by a method different from that described in the preceding
section.
In this section, we prepare several lemmas which are necessary in the next section.
First, the following lemma is a slight modification of \cite[Lemma 4.2]{KM4} 
which can be proved similarly.

\begin{lemma}\label{L-4-2}
Let $\{ P_{m}\},\ \{Q_{m}\},\ \{R_{m}\}$ be sequences such that
\begin{equation*}
P_{m}=\sum_{j=0}^{[m/2]}R_{m-2j} \frac{(i\pi)^{2j}}{(2j)!},\ Q_{m}=\sum_{j=0}^{[m/2]}R_{m-2j}\frac{(i\pi)^{2j}}{(2j+1)!}
\end{equation*}
for any $m \in \mathbb{N}_0$. Then  
\begin{align}
& P_{m}=-2\sum_{\tau=0}^{[m/2]}\zeta(2\tau)Q_{m-2\tau}, \label{P-2h} \\
& Q_{2h}=\frac{2}{\pi^2}\sum_{\tau=0}^{h}\left(2^{2h-2\tau+2}-1\right)\zeta(2h-2\tau+2)P_{2\tau} \label{Q-2h}
\end{align}
for any $h \in \mathbb{N}_0$. 
\end{lemma}

An important key to the argument in the next section is the following.

\begin{lemma}[\cite{KMTJC}\ Lemma 6.3] \label{L-4-3} Let $h \in \mathbb{N}$, 
$\lambda_j=(1+(-1)^j)/2$ for $j\in\mathbb{Z}$ and
\begin{align*}
& {\frak C}:=\left\{ C(l) \in \mathbb{C}\,|\, l \in \mathbb{Z},\ l \not=0 \right\}, \\
& {\frak D}:=\left\{ D(N;m;\eta) \in \mathbb{R}\,|\, N,m,\eta \in \mathbb{Z},\ N \not=0,\ m \geq 0,\ 1 \leq \eta \leq h\right\}, \\
& {\frak A}:=\{ a_\eta \in \mathbb{N}\,|\,1 \leq \eta \leq h\}
\end{align*}
be sets of numbers indexed by integers. Assume that the infinite series appearing in 
\begin{align}
\sum_{N \in \mathbb{Z} \atop N \not=0}(-1)^{N}C(N)e^{iN\theta} & -2\sum_{\eta=1}^{h}\sum_{k=0}^{a_\eta}\phi(a_\eta-k)\lambda_{a_\eta-k} \label{eq-4-3} \\
& \ \ \times \sum_{\xi=0}^{k}\left\{ \sum_{N \in \mathbb{Z} \atop N \not=0}(-1)^N D(N;k-\xi;\eta)e^{iN\theta}\right\}\frac{(i\theta)^\xi}{\xi!} \notag
\end{align}
are absolutely convergent for $\theta \in [-\pi,\pi]$, and that (\ref{eq-4-3}) is a constant function for $\theta \in [-\pi,\pi]$. Then, for $d \in \mathbb{N}_0$, 
\begin{align}
& \sum_{N \in \mathbb{Z} \atop N \not= 0}\frac{(-1)^{N}C(N)e^{iN\theta}}{N^d} -2\sum_{\eta=1}^{h}\sum_{k=0}^{a_\eta}\phi(a_\eta-k)\lambda_{a_\eta-k} \label{eq-4-4} \\
& \ \ \ \ \ \times \sum_{\xi=0}^{k}\bigg\{ \sum_{\omega=0}^{k-\xi}\binom{\omega+d-1}{\omega}(-1)^{\omega} \sum_{m \in \mathbb{Z} \atop m \not= 0}\frac{(-1)^m D(m;k-\xi-\omega;\eta)e^{im\theta}}{m^{d+\omega}}\bigg\}\frac{(i\theta)^\xi}{\xi!} \notag \\
& \ +2\sum_{k=0}^{d}\phi(d-k)\lambda_{d-k} \sum_{\xi=0}^{k}\bigg\{ \sum_{\eta=1}^{h} \sum_{\omega=0}^{a_\eta-1}\binom{\omega+k-\xi}{\omega}(-1)^{\omega}\notag \\
& \hspace{1in} \times \sum_{m \in \mathbb{Z} \atop m \not=0}\frac{D(m;a_\eta-1-\omega;\eta)}{m^{k-\xi+\omega+1}}\bigg\}\frac{(i\theta)^\xi}{\xi!} =0 \notag
\end{align}
holds for $\theta \in [-\pi,\pi]$, where the infinite series appearing on the left-hand side of (\ref{eq-4-4}) are absolutely convergent for $\theta \in [-\pi,\pi]$.
\end{lemma}

We prepare another lemma of the same feature, which is a slight generalization 
of \cite[Lemma 4.4]{KM4}. 

\begin{lemma} \label{L-4-4} Let $h \in \mathbb{N}$, 
\begin{align*}
& {\mathcal{A}}:=\left\{ \alpha(l) \in \mathbb{C}\,|\, l \in \mathbb{Z},\ l \not=0 \right\}, \\
& {\mathcal{B}}:=\left\{ \beta(N;m;\eta) \in \mathbb{R}\,|\, N,m,\eta \in \mathbb{Z},\ N \not=0,\ m \geq 0,\ 1 \leq \eta \leq h\right\}, \\
& {\mathcal{C}}:=\{ c_\eta \in \mathbb{N}\,|\,1 \leq \eta \leq h\}
\end{align*}
be sets of numbers indexed by integers, and 
\begin{equation}
\begin{split}
R_{\pm}(\theta)=\sum_{m \in \mathbb{Z} \atop m \not=0}(\pm i)^{m}\alpha(m)e^{im\theta/2} & -2\sum_{\eta=1}^{h}\sum_{k=0}^{c_\eta}\phi(c_\eta-k)\lambda_{c_\eta-k}  \\
& \ \ \times \sum_{\xi=0}^{k}\left\{ \sum_{m \in \mathbb{Z} \atop m \not=0}(\pm i)^m \beta(m;k-\xi;\eta)e^{im\theta/2}\right\}\frac{(i\theta)^\xi}{\xi!}.
\end{split}
\label{eq-4-10}
\end{equation}
Assume that both of the right-hand sides of $R_{\pm}(\theta)$ in \eqref{eq-4-10} are absolutely convergent for $\theta \in [-\pi,\pi]$, and that both $R_{+}(\theta)$ and $R_{-}(\theta)$ are constant functions on $[-\pi,\pi]$. Then, for $d \in \mathbb{N}$, 
\begin{align}
\sum_{m \in \mathbb{Z} \atop m \not= 0}\frac{\alpha(m)}{m^{d}} & \ -2\sum_{\eta=1}^{h}\sum_{k=0}^{[c_\eta/2]}\zeta(2k) \sum_{\omega=0}^{c_\eta-2k}\binom{\omega+d-1}{\omega}(-2)^{\omega} \sum_{m \in \mathbb{Z} \atop m \not= 0}\frac{\beta(m;c_\eta-2k-\omega;\eta)}{m^{d+\omega}} \label{eq-4-11} \\
& \ +2\sum_{\eta=1}^{h} \sum_{k=0}^{[d/2]}\zeta(2k)2^{-2k} \sum_{\omega=0}^{c_\eta-1}\binom{\omega+d-2k}{\omega}(-2)^{\omega} \notag\\
& \hspace{1in} \times \sum_{m \in \mathbb{Z} \atop m \not=0}\frac{((-1)^m+1)\beta(m;c_\eta-1-\omega;\eta)}{m^{d-2k+\omega+1}} \notag\\
& \ -2\sum_{\eta=1}^{h} \sum_{k=0}^{[(d+1)/2]}\zeta(2k)\left(1-2^{-2k}\right) \sum_{\omega=0}^{c_\eta-1}\binom{\omega+d-2k}{\omega}(-2)^{\omega} \notag\\
& \hspace{1in} \times \sum_{m \in \mathbb{Z} \atop m \not=0}\frac{((-1)^m-1)\beta(m;c_\eta-1-\omega;\eta)}{m^{d-2k+\omega+1}}=0 \notag
\end{align}
for $\theta \in [-\pi,\pi]$, where the infinite series appearing on the left-hand side of (\ref{eq-4-11}) are absolutely convergent for $\theta \in [-\pi,\pi]$.\end{lemma}

\begin{proof}
We just indicate how to modify the proof of \cite[Lemma 4.4]{KM4} to obtain the 
above lemma.   Let $\mathcal{G}_N^\pm(\theta)$ and $\mathfrak{C}_n^\pm$ be as in
the proof of \cite[Lemma 4.4]{KM4}. 
Putting $N=d+1$ for $d \in \mathbb{N}$ and $\theta=\pi$ in 
\cite[(4.11)]{KM4}, we obtain 
\begin{equation}
\frac{i^{d}}{2\pi}\left\{ \mathcal{G}_{d+1}^+(\pi)-\mathcal{G}_{d+1}^+(-\pi)\right\}=\sum_{\nu=0}^{[d/2]}\mathfrak{C}_{d-2\nu}^+ 2^{d-2\nu}\frac{(i\pi)^{2\nu}}{(2\nu+1)!}. \label{eq-4-15}
\end{equation}
Similarly we have
\begin{equation}
\frac{i^{d+1}}{2}\left\{ \mathcal{G}_{d+1}^+(\pi)+\mathcal{G}_{d+1}^+(-\pi)\right\}=\sum_{\nu=0}^{[(d+1)/2]}\mathfrak{C}_{d+1-2\nu}^+ 2^{d+1-2\nu}\frac{(i\pi)^{2\nu}}{(2\nu)!}. \label{eq-4-16}
\end{equation}
These are analogues of \cite[(4.12)]{KM4} and \cite[(4.13)]{KM4}, with replacing
$2d$ by $d$ and $d+1$, respectively.
By Lemma \ref{L-4-2}, we have
\begin{equation}
\begin{split}
& \frac{i^{d}}{2}\left\{ \mathcal{G}_{d}^+(\pi)+\mathcal{G}_{d}^+(-\pi)\right\} \\
& \ \ =-\frac{i^{d}}{\pi}\sum_{\tau=0}^{[d/2]}\zeta(2\tau)(-1)^\tau \left\{ \mathcal{G}_{d+1-2\tau}^+(\pi)-\mathcal{G}_{d+1-2\tau}^+(-\pi)\right\}, 
\end{split}
\label{eq-4-17}
\end{equation}
and similarly,
\begin{align}\label{eq-4-17b}                 
& \frac{1}{2}\left\{ \mathcal{G}_{d}^-(\pi)-\mathcal{G}_{d}^-(-\pi)\right\} \\          
& \ \ =\frac{1}{\pi}\sum_{\rho=0}^{[(d-1)/2]}\left(2^{2\rho+2}-1\right)
\zeta(2\rho+2)(-1)^\rho \left\{ \mathcal{G}_{d-1-2\rho}^-(\pi)+
\mathcal{G}_{d-1-2\rho}^-(-\pi)\right\}\notag\\    
& \ \ =-\frac{1}{\pi}\sum_{\tau=1}^{[(d+1)/2]}\left(2^{2\tau}-1\right)\zeta(2\tau)
(-1)^\tau \left\{ \mathcal{G}_{d+1-2\tau}^-(\pi)+\mathcal{G}_{d+1-2\tau}^-(-\pi)
\right\}.\notag
\end{align}
We evaluate each side of \eqref{eq-4-17} in the same way as in the proof of
\cite[Lemma 4.4]{KM4} with obvious modifications.   The result is almost the same
as \cite[(4.17)]{KM4}, just replacing $2d$ by $d$, and the summation $\sum_{\xi=0}^d$
by $\sum_{\xi=0}^{[d/2]}$.
Similarly from \eqref{eq-4-17b} we obtain a formula almost the same as
\cite[(4.19)]{KM4}, just replacing $2d$ by $d$, and the summation $\sum_{\xi=0}^d$
by $\sum_{\xi=0}^{[(d+1)/2]}$.
Combining those two formulas we obtain \eqref{eq-4-11}.
\end{proof}

\ 

\section{A functional relation corresponding to $I=\{2\}$}
\label{sec-4}

Using the lemmas in the previous section, we now construct a functional relation 
among $\zeta_2({\boldsymbol{s}};G_2)$, $\zeta(s)$ and
$\phi(s)=(2^{1-s}-1)\zeta(s)$, which correspond to the case $I=\{ 2\}$ 
in Section \ref{sec-2}. 

\begin{theorem} \label{Fn-Rel} 
For $p,q,r,u,v \in \mathbb{N}$,
\begin{equation}
\begin{split}
& \zeta_2(p,s,q,r,u,v;G_2)+(-1)^p \zeta_2(p,q,s,r,v,u;G_2)+(-1)^{p+q}\zeta_2(v,q,r,s,p,u;G_2) \\
& \ \ +(-1)^{p+q+v}\zeta_2(v,r,q,s,u,p;G_2)+(-1)^{p+q+r+v}\zeta_2(u,r,s,q,v,p;G_2)\\
& \ \ +(-1)^{p+q+r+u+v}\zeta_2(u,s,r,q,p,v;G_2)\\
& \ \ +I_1+I_2+\cdots+I_8=0
\end{split}
\label{Fn-G2}
\end{equation}
holds for all $s \in \mathbb{C}$ except for singularities of functions on the 
left-hand side, where $I_j$ $(1\leq j\leq 8)$, defined below, are linear combinations 
of $\zeta(s)$ and $\phi(s)$.
\end{theorem}

The definition of $I_j$ is given by
$$
I_j=A_j+B_{1j}+B_{2j} \qquad  (1\leq j\leq 8),
$$
where $A_j$, $B_{1j}$, $B_{2j}$ ($1\leq j\leq 8$) are defined as follows:
\begin{align*}
A_j&=2(-1)^{p+a_1}\sum_{k=0}^{[a_2/2]}\zeta(2k)\sum_{\sigma=0}^{a_2-2k}
\binom{\sigma+v-1}{\sigma}\sum_{\rho=0}^{a_3-a_4}
\binom{\rho+u-a_5}{\rho}\\
&\times\sum_{\omega=0}^{a_6-a_7}\binom{\omega+r-a_8}{\omega}
\binom{p+q-1-\omega-a_7}{a_9-1}\\
&\times(-1)^{a_{10}} 2^{a_{11}} 3^{a_{12}}
\zeta(s+p+q+r+u+v-2k),
\end{align*}
\begin{align*}
B_{1j}&=2(-1)^{p+b_1}\sum_{k=0}^{[v/2]}2^{-2k}\zeta(2k)\sum_{\sigma=0}^{a_2-1}
\binom{\sigma+v-2k}{\sigma}\sum_{\rho=0}^{a_3-b_4}
\binom{\rho+u-b_5}{\rho}\\                                                               
&\times\sum_{\omega=0}^{a_6-b_7}\binom{\omega+r-b_8}{\omega}                             
\binom{p+q-1-\omega-b_7}{a_9-1}(-1)^{a_{10}} 2^{b_{11}} 3^{b_{12}}\\                    
&\times\{\zeta(s+p+q+r+u+v-2k)+\phi(s+p+q+r+u+v-2k)\},                  
\end{align*}
and
\begin{align*}                                                                           
B_{2j}&=2(-1)^{p+b_1}\sum_{k=0}^{[(v+1)/2]}(1-2^{-2k})\zeta(2k)\sum_{\sigma=0}^{a_2-1}
\binom{\sigma+v-2k}{\sigma}\sum_{\rho=0}^{a_3-b_4}                                       
\binom{\rho+u-b_5}{\rho}\\                                                               
&\times\sum_{\omega=0}^{a_6-b_7}\binom{\omega+r-b_8}{\omega}                             
\binom{p+q-1-\omega-b_7}{a_9-1}(-1)^{a_{10}} 2^{b_{11}} 3^{b_{12}}\\                     
&\times\{\zeta(s+p+q+r+u+v-2k)-\phi(s+p+q+r+u+v-2k)\},                                  
\end{align*}
where $a_l=a_l(j)$, $b_l=b_l(j)$ are as follows: According to
$j=1,\ldots,8$, $a_l$ ($1\leq l\leq 12$) take the values
\begin{align*}
&a_1=1,v+1,1,v,v+1,v,v,v+1,\\
&a_2=p,u,q,u,r,u,r,u,\\
&a_3=p,p,q,q,r,r,r,r,\\
&a_4=2k+\sigma,1,2k+\sigma,1,2k+\sigma,1,2k+\sigma,1,\\
&a_5=1,2k+\sigma,1,2k+\sigma,1,2k+\sigma,1,2k+\sigma,\\
&a_6=p,p,q,q,p,p,q,q,\\
&a_7=2k+\sigma+\rho,1+\rho,2k+\sigma+\rho,1+\rho,1,1,1,1,\\
&a_8=1,1,1,1,2k+\sigma+\rho,1+\rho,2k+\sigma+\rho,1+\rho,\\
&a_9=q,q,p,p,q,q,p,p,\\
&a_{10}=0,0,\sigma+\rho+\omega,\rho+\omega,\rho,\rho,\rho+\omega,\rho+\omega,
\end{align*}
\begin{align*}
&a_{11}=\sigma-r-\omega,\sigma-r-\omega,\sigma-u-\rho,-u+2k+2\sigma-\rho-1,\\
&\qquad\qquad -r+2k+2\sigma+\rho-\omega-1,-r+\sigma+\rho-\omega,\sigma,\sigma,
\end{align*}
and
$$
a_{12}=-u-v-\sigma-\rho,-u-v+2k-\rho-1,0,-v-\sigma,0,-v-\sigma,0,-v-\sigma.
$$
Next define $b_l$ ($l=1,4,5,7,8,11,12$).   First,
$$
b_1=1,v+1,0,v,v+1,v,v,v+1.
$$
The definitions of $b_4,b_5,b_7,b_8$ are similar to $a_4,a_5,a_7,a_8$, but
all $2k+\sigma$ are replaced by $1+\sigma$.   Finally,
\begin{align*}                                                                           
&b_{11}=\sigma-r-\omega,\sigma-r-\omega,\sigma-u-\rho,-u+2\sigma-\rho,\\            
&\qquad\qquad -r+2\sigma+\rho-\omega,-r+\sigma+\rho-\omega,\sigma,\sigma,           
\end{align*}
and
\begin{align*}
&b_{12}=-u-v+2k-\sigma-\rho-1,-u-v+2k-\rho-1,0,-v+2k-\sigma-1,\\
&\qquad\qquad 0,-v+2k-\sigma-1,0,-v+2k-\sigma-1.
\end{align*}

\begin{remark}
When $p,q,r,u,v$ are even, formula \eqref{Fn-G2} coincides with our previous result 
given in \cite[Theorem 5.1]{KM4}.
On this occasion we correct some misprints in the statement of \cite[Theorem 5.1]{KM4}.
On line 8 of page 202, we should replace 
$\binom{2p+2q-2-\rho-\omega}{2q-1}$ by $\binom{2p+2q-2-\rho-\omega}{2p-1}$.
On lines 12 and 16 of page 203, we should replace $3^{-2v-\sigma}$ by 
$3^{-2v-\sigma-1+2k}$.
\end{remark}

\begin{proof}[Proof of Theorem \ref{Fn-Rel}]
The technique to prove this theorem is essentially the same as in our previous papers (see \cite[Section 5]{KM4}; also \cite[Section 7]{KMTJC}, \cite[Section 5]{KMTLie}). 
Hence it is enough to give a sketch of the proof here. 

From \cite[Lemma 5.3]{KMTLie}, we have
\begin{align}
& \sum_{l\not=0,\,m\geq 1 \atop {l+m\not=0 \atop l+2m\not=0}} \frac{(-1)^{l+m}x^m e^{i(l+m)\theta}}{l^{p}m^{s}(l+m)^{q}} \label{Lie}\\
& \ \ -2\sum_{j=0}^{p}\ \phi(p-j)\varepsilon_{p-j} \sum_{\xi=0}^{j}\binom{q-1+j-\xi}{q-1}(-1)^{j-\xi}\sum_{m=1}^\infty \frac{(-1)^{m}x^m e^{im\theta}}{m^{s+q+j-\xi}}\frac{(i\theta)^{\xi}}{\xi!} \notag\\
& \ \ +2\sum_{j=0}^{q}\ \phi(q-j)\varepsilon_{q-j} \sum_{\xi=0}^{j}\binom{p-1+j-\xi}{p-1}(-1)^{p-1}\sum_{m=1}^\infty \frac{x^m }{m^{s+p+j-\xi}}\frac{(i\theta)^{\xi}}{\xi!}\notag\\
& \quad =-\frac{(-1)^{p+q}}{2^p}\sum_{m=1}^\infty \frac{x^m}{m^{s+p+q}}\notag
\end{align}
for $p,q\in \mathbb{N}$, $\theta \in [-\pi,\pi]$, $s \in \mathbb{R}$ with $s>1$ and $x\in \mathbb{C}$ with $|x|= 1$. By Lemma \ref{L-4-3} with $d=r \in \mathbb{N}$, we have
\begin{align*}
& \sum_{l\in \mathbb{Z},\,l\not=0 \atop {m\geq 1 \atop {l+m\not=0 \atop l+2m\not=0}}} \frac{(-1)^{l}x^m e^{i(l+2m)\theta}}{l^{p}m^s(l+m)^{q}(l+2m)^{r}}\\
& \ -2\sum_{j=0}^{p}\ \phi(p-j)\ \la_{p-j}\sum_{\xi=0}^{j} \sum_{\omega=0}^{j-\xi} \binom{\omega+r-1}{\omega} (-1)^\omega  \\
& \hspace{0.5in} \times \binom{q-1+j-\xi-\omega}{b-1}(-1)^{j-\xi-\omega}\frac{1}{2^{r+\omega}} \sum_{m=1}^\infty \frac{x^m e^{2im\theta}}{m^{s+q+j-\xi+r}}\frac{(i\theta)^{\xi}}{\xi!}  \\
& \ \ +2\sum_{j=0}^{q}\ \phi(q-j)\ \la_{q-j}\sum_{\xi=0}^{j} \sum_{\omega=0}^{j-\xi} \binom{\omega+r-1}{\omega} (-1)^\omega  \\
& \hspace{0.5in} \times \binom{p-1+j-\xi-\omega}{p-1}(-1)^{p-1}\sum_{m=1}^\infty \frac{(-1)^m x^m e^{im\theta}}{m^{s+p+r+j-\xi}}\frac{(i\theta)^{\xi}}{\xi!}  \\
& \ \ +2\sum_{j=0}^{r}\ \phi(r-j)\ \la_{r-j}\sum_{\xi=0}^{j} \sum_{\omega=0}^{p-1} \binom{\omega+j-\xi}{\omega} (-1)^\omega  \\
& \hspace{0.5in} \times \binom{p+q-2-\omega}{q-1}(-1)^{p-1-\omega}\frac{1}{2^{j-\xi+\omega+1}} \sum_{m=1}^\infty \frac{x^m}{m^{s+p+q+j-\xi}}\frac{(i\theta)^{\xi}}{\xi!}  \\
& \ \ -2\sum_{j=0}^{r}\ \phi(r-j)\ \la_{r-j}\sum_{\xi=0}^{j} \sum_{\omega=0}^{q-1} \binom{\omega+j-\xi}{\omega} (-1)^\omega  \\
& \hspace{0.5in} \times \binom{p+q-2-\omega}{p-1}(-1)^{p-1} \sum_{m=1}^\infty \frac{x^m}{m^{s+p+q+j-\xi}}\frac{(i\theta)^{\xi}}{\xi!}=0 
\end{align*}
for $\theta \in [-\pi,\pi]$, $p,q,r \in \mathbb{N}$, $s\in \mathbb{R}$ with $s>1$ and $x\in \mathbb{C}$ with $|x|\leq 1$. Here we replace $x$ by $-xe^{i\theta}$ and move the terms corresponding to $l+3m=0$ of the first member on the left-hand side of the above equation to the right-hand side. Then we have
\begin{align*}
& \sum_{l\in \mathbb{Z},\,l\not=0 \atop {m \geq 1 \atop {l+m\not=0 \atop {l+2m \not= 0 \atop l+3m\not= 0}}}} \frac{(-1)^{l+m}x^m e^{i(l+3m)\theta}}{l^{p}m^s(l+m)^{q}(l+2m)^{r}}\\
& \ \ -2\sum_{j=0}^{p}\ \phi(p-j)\ \la_{p-j}\sum_{\xi=0}^{j} \sum_{\omega=0}^{j-\xi} \binom{\omega+r-1}{\omega} (-1)^\omega  \\
& \hspace{0.5in} \times \binom{q-1+j-\xi-\omega}{q-1}(-1)^{j-\xi-\omega}\frac{1}{2^{r+\omega}} \sum_{m=1}^\infty \frac{(-1)^m x^m e^{3im\theta}}{m^{s+q+r+j-\xi}}\frac{(i\theta)^{\xi}}{\xi!}  \\
& \ \ +2\sum_{j=0}^{q}\ \phi(q-j)\ \la_{q-j}\sum_{\xi=0}^{j} \sum_{\omega=0}^{j-\xi} \binom{\omega+r-1}{\omega} (-1)^\omega  \\
& \hspace{0.5in} \times \binom{p-1+j-\xi-\omega}{p-1}(-1)^{p-1}\sum_{m=1}^\infty \frac{x^m e^{2im\theta}}{m^{s+p+r+j-\xi}}\frac{(i\theta)^{\xi}}{\xi!}  \\
& \ \ +2\sum_{j=0}^{r}\ \phi(r-j)\ \la_{r-j}\sum_{\xi=0}^{j} \sum_{\omega=0}^{p-1} \binom{\omega+j-\xi}{\omega} (-1)^\omega  \\
& \hspace{0.5in} \times \binom{p+q-2-\omega}{q-1}(-1)^{p-1-\omega}\frac{1}{2^{j-\xi+\omega+1}} \sum_{m=1}^\infty \frac{(-1)^m x^m e^{im\theta}}{m^{s+p+q+j-\xi}}\frac{(i\theta)^{\xi}}{\xi!}  \\
& \ \ +2\sum_{j=0}^{r}\ \phi(r-j)\ \la_{r-j}\sum_{\xi=0}^{j} \sum_{\omega=0}^{q-1} \binom{\omega+j-\xi}{\omega} (-1)^\omega  \\
& \hspace{0.5in} \times \binom{p+q-2-\omega}{p-1} (-1)^{p-1}\sum_{m=1}^\infty \frac{(-1)^m x^m e^{im\theta}}{m^{s+p+q+j-\xi}}\frac{(i\theta)^{\xi}}{\xi!} \\
& =- \frac{(-1)^{p+q+r}}{3^{2p}2^{2q}}\sum_{m =1}^\infty \frac{x^m }{m^{s+p+q+r}}.
\end{align*}
We again apply Lemma \ref{L-4-3} with $d=u\in \mathbb{N}$ to the above equation. 
Then we have
\begin{equation}
\begin{split}
& \sum_{l\in \mathbb{Z},\, l\not=0 \atop{m \geq 1 \atop {l+m\not=0 \atop {l+2m \not= 0 \atop l+3m\not= 0}}}} \frac{(-1)^{l+m}x^m e^{i(l+3m)\theta}}{l^{p}m^s(l+m)^{q}(l+2m)^{r}(l+3m)^{u}}\\
& \ \ +J_1(\theta;x)+J_2(\theta;x)+J_3(\theta;x)+J_4(\theta;x)=0,
\end{split}
\label{eq-4-7} 
\end{equation}
where
\begin{align*}
& J_1 (\theta;x) \\
& =-2\sum_{j=0}^{p}\ \phi(p-j)\ \la_{p-j}\sum_{\xi=0}^{j} \sum_{\rho=0}^{j-\xi}\binom{\rho+u-1}{\rho}(-1)^\rho \sum_{\omega=0}^{j-\xi-\rho} \binom{\omega+r-1}{\omega} (-1)^\omega  \\
& \hspace{0.2in} \times 3^{-u-\rho}\binom{q-1+j-\xi-\rho-\omega}{q-1}\frac{(-1)^{j-\xi-\rho-\omega}}{2^{r+\omega}} \sum_{m=1}^\infty \frac{(-1)^m x^m e^{3im\theta}}{m^{s+q+r+u+j-\xi}}\frac{(i\theta)^{\xi}}{\xi!} \\
& +2\sum_{j=0}^{u}\ \phi(u-j)\ \la_{u-j}\sum_{\xi=0}^{j} \sum_{\rho=0}^{p-1}\binom{\rho+j-\xi}{\rho}(-1)^\rho \sum_{\omega=0}^{p-1-\rho} \binom{\omega+r-1}{\omega} (-1)^\omega  \\
& \hspace{0.2in} \times 3^{-j+\xi-\rho-1}\binom{p+q-2-\rho-\omega}{q-1}\frac{(-1)^{p-1-\rho-\omega}}{2^{r+\omega}} \sum_{m=1}^\infty \frac{x^m}{m^{s+p+q+r+2j-\xi}}\frac{(i\theta)^{\xi}}{\xi!}.
\end{align*}
We can similarly write $J_2(\theta;x)$, $J_3(\theta;x)$ and $J_4(\theta;x)$, but they are omitted for the purpose of saving space. 

Next, setting $x=\pm i e^{-3i\theta/2}$ in \eqref{eq-4-7} and 
moving the terms corresponding to $2l+3m=0$ of the first member on the left-hand side to the right-hand side, we have
\begin{equation}
\begin{split}
& \sum_{l\in \mathbb{Z},\,l\not=0 \atop{m \geq 1 \atop {l+m\not=0 \atop {l+2m \not= 0 \atop {l+3m\not= 0 \atop 2l+3m \not=0}}}}} \frac{(-1)^{l+m} (\pm i)^{m} e^{i(2l+3m)\theta/2}}{l^{p}m^s(l+m)^{q}(l+2m)^{r}(l+3m)^{u}}\\
& \ \ +J_1(\theta;\pm i e^{-3i\theta/2})+J_2(\theta;\pm i e^{-3i\theta/2})+J_3(\theta;\pm i e^{-3i\theta/2})+J_4(\theta;\pm i e^{-3i\theta/2}) \\
& \ \ =-\sum_{l,m =1 \atop 2l=3m}^\infty \frac{1}{(-l)^{p}m^s(-l+m)^{q}(-l+2m)^{r}(-l+3m)^{u}}.
\end{split}
\label{eq-4-9} 
\end{equation}
Note that $(-1)^{l+m} (\pm i)^{m}=(\pm i)^{2l+3m}$. Hence 
we can apply Lemma \ref{L-4-4} with $d=v\in \mathbb{N}$ to \eqref{eq-4-9}, because  
we can see that the left-hand side of \eqref{eq-4-9} is of the same form 
as the right-hand side of \eqref{eq-4-10}. 
Consequently we obtain the equation given from \eqref{eq-4-11}. 
The first term on the left-hand side of the obtained equation is
\begin{equation}\label{eqeq}
\sum_{l\in \mathbb{Z},\ l\not=0 \atop {m \geq 1 \atop {l+m\not=0 \atop {l +2m\not= 0 
\atop {l+3m\not= 0 \atop 2l+3m \not=0}}}}} \frac{1}{l^{p}m^s(l+m)^{q}(l+2m)^{r}
(l+3m)^{u}(2l+3m)^{v}},
\end{equation}
while the remaining terms on the left-hand side of the obtained equation
can be expressed explicitly in terms of the Riemann zeta-function, which are
$I_1+\cdots+I_8$ in the statement of the theorem.
On the other hand, we see that \eqref{eqeq} is equal to
\begin{align}
& \zeta_2(p,s,q,r,u,v;G_2)+(-1)^p \zeta_2(p,q,s,c,v,u;G_2)+(-1)^{p+q}\zeta_2(v,q,r,s,p,u;G_2) \notag\\
& \ \ +(-1)^{p+q+v}\zeta_2(v,r,q,s,u,p;G_2)+(-1)^{p+q+r+v}\zeta_2(u,r,s,q,v,a;G_2)\notag\\
& \ \ +(-1)^{p+q+r+u+v}\zeta_2(u,s,r,q,p,v;G_2).\label{5-4-1}
\end{align}
This can be shown by decomposing \eqref{eqeq} by the same
argument as in \cite[Section 7]{KMTJC}; or, 
since \eqref{eqeq} coincides with $S(p,s,q,r,u,v),{\bf 0};\{2\};G_2)$ 
(see \eqref{eq:def_S}), 
\eqref{5-4-1} simply follows from \eqref{eq:func_eq}. 
Thus we obtain the assertion of the theorem.
\end{proof}

Setting $(p,q,r,u,v)=(2a,b,2c-1,d,d)$ for $a,b,c,d\in \mathbb{N}$ in \eqref{Fn-G2}, we see that
\begin{align}
& \ \ \zeta_{2}(2a,s,b,2c-1,d,d;G_2)+\zeta_{2}(2a,b,s,2c-1,d,d;G_2)  \notag\\
& +(-1)^b \zeta_{2}(d,b,2c-1,s,2a,d;G_2)+(-1)^{b+d}\zeta_{2}(d,2c-1,b,s,d,2a;G_2) \notag\\
& -(-1)^{b+d}\zeta_{2}(d,2c-1,s,b,d,2a;G_2)-(-1)^{b}\zeta_{2}(d,s,2c-1,b,2a,d;G_2) \label{5-4-2}
\end{align}
is expressed in terms of $\zeta(s)$ and $\phi(s)$. 
As we noted above (see \eqref{5-4-1}), \eqref{5-4-2} coincides with 
$$S((2a,s,b,2c-1,d,d),{\bf 0};\{2\};G_2).$$
In particular when $s=b$, it is equal to $2\zeta_2(2a,b,b,2c-1,d,d);G_2)$ (see \eqref{2-29}). Therefore we have the following.

\begin{corollary}\label{C-3-2} For $a,b,c,d\in \mathbb{N}$,
\begin{equation}
\zeta_{2}(2a,b,b,2c-1,d,d;G_2)\in \mathbb{Q}[\{\zeta(j)\,|\,j\in \mathbb{N}_{\geq 2}\}]. \label{e-3-2-1}
\end{equation}
\end{corollary}

\begin{example}\label{E-3-3} 
Putting $(p,q,r,u,v)=(2,1,1,1,1)$ in \eqref{Fn-G2}, we have 
\begin{align*}
& \ \ \zeta_{2}(2,s,1,1,1,1;G_2)+\zeta_{2}(2,1,s,1,1,1;G_2)-\zeta_{2}(1,1,1,s,2,1;G_2)  \\
& +\zeta_{2}(1,1,1,s,1,2;G_2)  -\zeta_{2}(1,1,s,1,1,2;G_2)+\zeta_{2}(1,s,1,1,2,1;G_2) \notag\\
& -\frac{1}{9}\zeta(2)\zeta(s+4)+\frac{109}{648}\zeta(s+6)=0.
\end{align*}
Setting $s=1$, we obtain a special case of \eqref{e-3-2-1}, that is, 
\begin{align}\label{rei1}
\zeta_{2}(2,1,1,1,1,1;G_2)=\frac{1}{18}\zeta(2)\zeta(5)-\frac{109}{1296}\zeta(7),
\end{align}
which is \eqref{Zhao-1} noted in Section \ref{sec-1}. Similarly we can compute
\begin{align}
& \zeta_{2}(4,1,1,1,1,1;G_2) =\frac{1}{18}\zeta(4)\zeta(5)+\frac{145}{648}\zeta(2)\zeta(7)-\frac{19753}{46656}\zeta(9),\label{rei2}\\
& \zeta_2(2,1,1,1,2,2;G_2)=-\frac{187}{324}\zeta(2)\zeta(7)+\frac{11149}{11664}\zeta(9),\label{rei3}\\
&\zeta_{2}(4,2,2,1,1,1;G_2)=\frac{1}{18}\zeta(4)\zeta(7)+\frac{595}{648}\zeta(2)\zeta(9)-\frac{73201}{46656}\zeta(11),\label{rei4}\\
& \zeta_{2}(2,1,1,5,3,3;G_2) =\frac{5}{4}\zeta(4)\zeta(11)+\frac{1043857}{23328}\zeta(2)\zeta(13)-\frac{41971423}{559872}\zeta(15),\label{rei5}\\
& \zeta_2(4,2,2,1,4,4;G_2) =\frac{61441}{209952}\zeta(4)\zeta(13)+\frac{600677}{944784}\zeta(2)\zeta(15)-\frac{23172773}{17006112}\zeta(17),\label{rei6}\\
& \zeta_2(2,4,4,3,3,3;G_2) =\frac{1}{8}\zeta(4)\zeta(15)+\frac{281221}{23328}\zeta(2)\zeta(17)-\frac{11177971}{559872}\zeta(19).\label{rei7}
\end{align}
\end{example}

\section{The parity result for the zeta-function of $G_2$}\label{sec-final}

We conclude this paper with a discussion on the parity result for the zeta-function 
of $G_2$.

It is well-known that the double zeta values satisfy that 
$$\sum_{m,n=1}^\infty \frac{1}{m^p (m+n)^q}\in \mathbb{Q}[\{\zeta(j+1)\,|\,j\in \mathbb{N}\}],$$
for $p,q \in \mathbb{N}$ with $q\geq 2$ and $2\nmid(p+q)$, which was proved by Euler. The same situation holds for the zeta values of type $A_2$ (see \cite{To}) and type $B_2$ (see \cite{TsAr}):
\begin{align*}
\zeta_2(p,q,r;A_2), \ \zeta_2(t,u,v,w;B_2) \in \mathbb{Q}[\{\zeta(j+1)\,|\,j\in \mathbb{N}\}]
\end{align*}
for $p,q,r,t,u,v,w\in \mathbb{N}$ with $2\nmid(p+q+r)$ and $2\nmid(t+u+v+w)$. 
These may be regarded as examples of ``parity results''.
(In general, a ``parity result'' means a property that some multiple zeta value
whose weight and depth are of different parity can be written in terms of multiple zeta values
of lower depth.) 
Does the same type of assertion hold for
$\zeta_2(p,q,r,u,v,w;G_2)$?
It seems that the answer is negative; 
in view of Example \ref{Exam-i-1} (especially \eqref{VR-01}), we find that the 
following modified statement is more plausible:
\begin{align}\label{parity}
\zeta_2(p,q,r,u,v,w;G_2) \in 
\mathbb{Q}[\{\zeta(j+1),\ L(j,\chi_3)\,|\,j\in \mathbb{N}\}]
\end{align}
for $p,q,r,u,v,w\in \mathbb{N}$ with $2\nmid(p+q+r+u+v+w)$. 

In this direction, an interesting result was given in Okamoto's paper \cite{Ok}.
Inspired by the work of Nakamura \cite{Nak08b} and Onodera \cite{On},
Okamoto proved (his Theorems 2.3 and 4.5) that the values of certain generalized 
double zeta-functions, including the case
$\zeta_2(p,q,r,u,v,w;G_2)$ with $2\nmid(p+q+r+u+v+w)$, can be expressed in terms of
the Riemann zeta values and the values of Clausen-type functions, that is,
$$
S_r(x)=\sum_{m\geq 1}\frac{\sin(2\pi mx)}{m^r}\quad{\rm or} \quad
C_r(x)=\sum_{m\geq 1}\frac{\cos(2\pi mx)}{m^r}\qquad(r\in\mathbb{N})
$$
with $x=j/l\in\mathbb{Q}$  ($l\in \mathbb{N}$, $0\leq j<l$).
Moreover in his formula, in the case of $G_2$, only the cases $l=1,2,3,4,6$ and $12$ of 
Clausen-type functions appear.   For $l=1,2,3$ and $6$, the values 
$S_r(j/l)$ and $C_r(j/l)$ can be written in terms of
the values of $\zeta(s)$ and $L(s,\chi_3)$, similarly to \eqref{cosclausen} and
\eqref{sinclausen}.   Therefore, Okamoto's result implies that
{\it if} $\;2\nmid(p+q+r+u+v+w)$, {\it then the value} $\;\zeta_2(p,q,r,u,v,w;G_2)$
{\it can be written in terms of} $\;\zeta(s)$, $L(s,\chi_3)$, $S_r(j/l)$ {\it and}
$\;C_r(j/l)$ {\it for} $\;l=4, 12$ and $0< j<l$ {\it with} $(j,l)=1$. 
This may be regarded as a kind of ``generalized parity result''.

If we apply Okamoto's theorem directly, we obtain a rather long expression of
special values in terms of Clausen-type functions.   But we have checked, using
PARI/GP, that his expression actually agrees with our expression for
\eqref{VR-01}, \eqref{rei1}, \eqref{rei2}, \eqref{rei3} and \eqref{rei4}.
To check the other examples (\eqref{rei5}, \eqref{rei6}, \eqref{rei7}) we would
require much more running time, so we did not check them.

Though only the values of $\zeta(s)$, $L(s,\chi_3)$ appear in all of our
examples, it is not sure whether $S_r(j/l)$ or $C_r(j/l)$ ($l=4, 12$; $0< j<l,$ $(j,l)=1$) will appear
or not (in other words, \eqref{parity} would hold or not) in some other examples. 

It seems that for the zeta-function of $G_2$, the parity result is valid only in
this generalized form.
On the other hand, our Example \ref{E-3-3} shows that
sometimes the value $\zeta_2(p,q,r,u,v,w;G_2)$, with $2\nmid(p+q+r+u+v+w)$, can be 
expressed only by the values of $\zeta(s)$.   It is an interesting problem to
determine when such a restricted form of the parity result holds.

\bigskip

\baselineskip 14pt

\bibliographystyle{amsplain}

\begin{thebibliography}{999}

\bibitem{Apostol}
T. M. Apostol, \emph{Introduction to Analytic Number Theory}, Springer, New York-Heidelberg, 1976.

\bibitem{Bourbaki}
N. Bourbaki, \emph{Groupes et Alg{\`e}bres de Lie, Chapitres 4, 5 et 6}, Hermann, Paris, 1968.



\bibitem{Hum72}
J. E. Humphreys, \emph{Introduction to Lie Algebras and Representation Theory}, 
Graduate Texts in Mathematics \textbf{9}, Springer-Verlag, New York-Berlin, 1972.

\bibitem{Hum}
J. E. Humphreys, \emph{Reflection Groups and Coxeter Groups}, Cambridge University 
Press, Cambridge, 1990.



\bibitem{KMT}
Y. Komori, K. Matsumoto and H. Tsumura, Zeta-functions of root systems, in ``Proceedings of the Conference on $L$-functions'' (Fukuoka, 2006), L. Weng and M. Kaneko (eds.), World Sci. Publ., Hackensack, NJ, 2007, pp.~115-140.

\bibitem{KMTpja}
Y. Komori, K. Matsumoto and H. Tsumura, Zeta and $L$-functions and Bernoulli polynomials of root systems, \emph{Proc.~Japan Acad. Ser.~A} \textbf{84} (2008), 57--62.

\bibitem{KM2}
Y. Komori, K. Matsumoto and H. Tsumura, 
On Witten multiple zeta-functions associated with semisimple Lie algebras II, \emph{J. Math.~Soc.~Japan} \textbf{62} (2010), 355--394. 

\bibitem{KMT-L}
Y. Komori, K. Matsumoto and H. Tsumura, 
On multiple Bernoulli polynomials and multiple $L$-functions of root systems, \emph{Proc.~London Math.~Soc.} \textbf{100} (2010), 303--347.

\bibitem{KMTJC}
Y. Komori, K. Matsumoto and H. Tsumura, 
Functional relations for zeta-functions of root systems, in ``Number Theory: Dreaming in Dreams - Proceedings of the 5th China-Japan Seminar'', T. Aoki, S. Kanemitsu and J. -Y. Liu (eds.), World Sci. Publ., Hackensack, NJ, 2010, pp.~135--183. 


\bibitem{KM3}
Y. Komori, K. Matsumoto and H. Tsumura, 
On Witten multiple zeta-functions associated with semisimple Lie algebras III,  in ``Multiple Dirichlet Series, L-functions and Automorphic Forms'', D. Bump, S. Friedberg and D. Goldfeld (eds.), \emph{Progr. Math.} \textbf{300}, Birkh\"auser/Springer, New York, 2012, pp. 223--286.

\bibitem{KM4}
Y. Komori, K. Matsumoto and H. Tsumura, 
On Witten multiple zeta-functions associated with semisimple Lie algebras IV, \emph{Glasgow Math. J.} \textbf{53} (2011), 185--206.

\bibitem{KMTLie}
Y. Komori, K. Matsumoto and H. Tsumura, 
Zeta-functions of weight lattices of compact connected semisimple Lie groups, preprint, arXiv:math/1011.0323.





\bibitem{MNOT}
K. Matsumoto, T. Nakamura, H. Ochiai and H. Tsumura, 
On value-relations, functional relations and singularities of Mordell-Tornheim and related triple zeta-functions, \emph{Acta Arith.} \textbf{132} (2008), 99--125.

\bibitem{MNT}
K. Matsumoto, T. Nakamura and H. Tsumura,
Functional relations and special values of Mordell-Tornheim triple zeta and 
$L$-functions,
\emph{Proc. Amer. Math. Soc.} \textbf{136} (2008), 2135-2145.

\bibitem{MTF}
K. Matsumoto and H. Tsumura, On Witten multiple zeta-functions associated with semisimple Lie algebras I, \emph{Ann.~Inst.~Fourier} \textbf{56} (2006), 1457--1504.

\bibitem{Nak06}
T. Nakamura, A functional relation for the Tornheim double zeta function,
\emph{Acta Arith.} \textbf{125} (2006), 257-263.

\bibitem{Nak08a}
T. Nakamura, Double Lerch series and their functional relations,
\emph{Aequationes Math.} \textbf{75} (2008), 251-259.

\bibitem{Nak08b}
T. Nakamura, Double Lerch value relations and functional relations for Witten zeta functions,
\emph{Tokyo J. Math.} \textbf{31} (2008), 551-574.

\bibitem{Ok}
T. Okamoto, Multiple zeta values related with the zeta-function 
of the root system of type $A_2$, $B_2$ and $G_2$, 
\emph{Comment. Math. Univ. St. Pauli} \textbf{61} (2012), 9-27.

\bibitem{On}
K. Onodera, Generalized log sine integrals and the Mordell-Tornheim zeta values,
\emph{Trans. Amer. Math. Soc.} \textbf{363} (2011), 1463-1485.

\bibitem{To}
  L. Tornheim, {Harmonic double series}, \emph{Amer.~J. Math.} \textbf{72} (1950), 303--314.


\bibitem{TsAr}
H. Tsumura, 
On Witten's type of zeta values attached to $SO(5)$, \emph{Arch.~Math.~(Basel)} \textbf{82} (2004), 147--152.


\bibitem{Wi}
{E. Witten}, {On quantum gauge theories in two dimensions}, \emph{Comm.~Math.~Phys.} \textbf{141} (1991), 153--209.

\bibitem{Za}
{D. Zagier}, Values of zeta functions and their applications, in `First European Congress of Mathematics' Vol.~II, A. Joseph
   et al. (eds.), \emph{Progr.~Math.} \textbf{120}, Birkh{\"a}user, Basel, 1994,
   pp.~497--512.

\bibitem{Za2}
{D. Zagier}, Introduction to multiple zeta values, 
Lectures at Kyushu Univ., 1999.

\bibitem{Zhao}
{J. Zhao}, 
Multi-polylogs at twelfth roots of unity and special values of Witten multiple zeta function attached to the exceptional Lie algebra $\mathfrak{g}_2$, \emph{J. Algebra Appl.} \textbf{9} (2010), 327--337.

\end{thebibliography}

\ 

\end{document}